\documentclass{article}
\usepackage{latexsym}
\usepackage{color}

\usepackage{amssymb}
\usepackage[utf8]{inputenc}
\usepackage[T1]{fontenc}
\usepackage{tikz}
\newcommand{\szkielet}{\draw[->] (0,0)--(6,0);  \draw[->] (0,0)--(0,5);
 \foreach \position in {(0,0),(1,0),(0,1),(2,0),(1,1),(0,2),(3,0),(2,1),(1,2),(0,3),(4,0),(3,1),(2,2),(1,3),(0,4),(5,0),(4,1),(3,2),(2,3),(1,4),(0,5)}  \fill \position circle (1.5pt);}

\newcommand{\szkiel}{\draw[->] (0,0)--(5,0);  \draw[->] (0,0)--(0,5);
 \foreach \position in {(0,0),(1,0),(0,1),(2,0),(1,1),(0,2),(3,0),(2,1),(1,2),(0,3),(4,0),(3,1),(2,2),(1,3),(0,4),(4,1),(3,2),(2,3),(1,4),(0,5)}  \fill \position circle (1.5pt);}

\newcommand{\Rn}{\mathbb{R}}

\newcommand{\Nn}{\mathbb{N}}
\newcommand{\Jac}{\mathop{\mathrm{Jac}}}
\newcommand{\supp}{\mathop{\mathrm{supp}}}
\newcommand{\conv}{\mathop{\mathrm{conv}}}
\newtheorem{Lemma}{Lemma}
\newtheorem{Theorem}{Theorem}

\newtheorem{Cor}{Corollary}
\newenvironment{proof}[1][Proof]{\textbf{#1.} }{\
\rule{0.5em}{0.5em}}

\begin{document}
\author{Janusz Gwoździewicz}
\title{Real Jacobian pairs with components of low degrees}
\maketitle

\begin{abstract}
We prove that every polynomial map $(f,g):\Rn^2\to\Rn^2$ with nowhere vanishing Jacobian 
such that $\deg f\leq 5$,  $\deg g \leq 6$ is injective.
\end{abstract}

\section{Introduction}

\renewcommand{\thefootnote}{}
\footnotetext{
     \noindent   \begin{minipage}[t]{4.5in}
       {\small  2010 \emph{Mathematics Subject Classification}:
       Primary 14R15.\\
       \emph{Key words and phrases}: real Jacobian conjecture, Newton polygon.}
       \end{minipage}}

Let $F=(f,g):\Rn^2\to\Rn^2$ be a polynomial map such that the Jacobi determinant 
$\Jac(f,g)=\frac{\partial f}{\partial x}\frac{\partial g}{\partial y}-
\frac{\partial f}{\partial y}\frac{\partial g}{\partial x}$ 
is a nowhere vanishing polynomial.  
The \emph{real Jacobian conjecture} states that $F$ is injective
and thus in view of~\cite{BB-R} a bijective polynomial mapping.

The real Jacobian conjecture was disproved by Pinchuk in~\cite{P}.
He constructed  a polynomial map $F:\Rn^2\to\Rn^2$ 
with everywhere positive jacobian which is not injective. The components of 
Pinchuk's mapping have degrees 10 and 40. Later on it was observed that a simple
modification of Pinchuk's counterexample decreases degrees of  components to 10 and 35. 
Up to now nobody found a counterexample to the real Jacobian conjecture with smaller 
degrees of components. 
Hence it is natural to ask a question: Under what additional assumptions on 
degrees the real Jacobian conjecture remains true?

In \cite{Gw1} it was proved that the real Jacobian conjecture is true under the assumption 
$\deg f\leq 3$, $\deg g\leq 3$.
Braun and Santos \cite{BF} generalized this result showing that it is enough to assume 
that  $\deg f\leq 3$ while degree of $g$ can be arbitrary.  
In~\cite{BO} Braun and  Or\'efice-Okamoto proved, that it is enough to assume $\deg f\leq 4$. 

In the present paper we continue this line of research. Its main result is

\begin{Theorem}\label{Tw:main}
Every polynomial map $(f,g):\Rn^2\to\Rn^2$ with nowhere vanishing Jacobian 
such that $\deg f\leq5$, $\deg g\leq 6$ is injective.
\end{Theorem}

\section{Lemmas}

In this section we present results that allow to prove smoothly Theorem~\ref{Tw:main}. 
Most of them are taken from \cite{BO,Gw1,Gw2}.  
Only Lemmas~\ref{L:znak0}, \ref{L:znak} and~\ref{L:edge} are new. 

\medskip
Let $f$, $g\in \Rn[x,y]$. 
We will say that $(f,g)$ is a \textit{Jacobian pair} if its Jacobi determinant 
$\Jac(f,g):=\frac{\partial f}{\partial x}\frac{\partial g}{\partial y}-
\frac{\partial f}{\partial y}\frac{\partial g}{\partial x}$ 
is everywhere positive or everywhere negative. 
A jacobian pair $(f,g)$ will be called \textit{typical} if the map $(f,g):\Rn^2\to\Rn^2$ is injective, 
otherwise it will be called an \textit{atypical jacobian pair}.


\begin{Lemma}\label{L:para}
Let $(\tilde f,\tilde g)=K\circ(f,g)\circ L$ for some invertible affine mappings 
 $L,K:\Rn^2\to \Rn^2$.
Then $(f,g)$ is a jacobian pair (resp.\ typical jacobian pair)  
if and only if $(\tilde f,\tilde g)$ is a jacobian pair (resp.\ typical jacobian pair).
\end{Lemma}

The proof is obvious. 

The proofs of the following well-known result can be found for example in~\cite[Lemma~1.2]{BO} 
or~\cite[Lemma 1]{Gw1}. 

\begin{Lemma}\label{L:poziomica}
Let $(f,g)$ be a jacobian pair. Then $(f,g)$ is a typical jacobian pair if and only if 
all level sets of $f$ are connected. 
\end{Lemma}

Let $(f,g)$ be an atypical jacobian pair. Consider a pencil of polynomials 
$\{\lambda f+\mu g\}_ {(\lambda,\mu)\in\Rn^2\setminus\{0,0)\}}$.  
It follows from~Lemma~\ref{L:para} that any two linearly independent polynomials of this 
pencil constitute an atypical jacobian pair.  By Lemma~\ref{L:poziomica} we get:

\begin{Cor}\label{W:1}
Let $(f,g)$ be an atypical jacobian pair. Then every polynomial $\lambda f+\mu g$ 
for $(\lambda,\mu)\neq (0,0)$ does not have critical points and has at least one 
disconnected level set. 
\end{Cor}

In order to state subsequent lemmas we need a few notions.
Let $f=\sum_{(i,j)} a_{ij}x^iy^j$ be a nonzero polynomial. 
By a \textit{support} $\supp(f)$ we mean the set of monomials 
that appear in $f$ with nonzero coefficient. We call the set
$$\Delta_f=\conv(\{\,(i,j)\in\Nn^2: a_{ij}\neq 0\})$$
the \textit{Newton polygon} of $f$. Here $\conv(A)$ denotes the convex hull of a set $A$.

For a compact subset $\Delta$ of $\Rn^2$ and $\xi\in\Rn^2$ we define 
$l(\Delta,\xi)=\max_{\alpha\in \Delta} \langle \xi,\alpha \rangle$ and 
$\Delta^{\xi}=\{\,\alpha\in \Delta: \langle \xi,\alpha \rangle = l(\Delta,\xi)\,\}$. 

For any subset  $E$ of $\Rn^2$ we call the polynomial $f|_E=\sum_{(i,j)\in E} a_{ij}x^iy^j$ 
the \textit{symbolic restriction} of $f$ to $E$.  
If $\xi\in\Rn^2$ is a nonzero vector then $\Delta_f^{\xi}$ is a vertex or an edge of the polygon 
$\Delta_f$. In this case $f^{\xi}:=f|_{\Delta_f^{\xi}}$ is a quasi-homogeneous polynomial of 
weighted degree $\mathrm{w}(f)= l(\Delta_f,\xi)$ provided that $\mathrm{w}(x)=\xi_1$ and $\mathrm{w}(y)=\xi_2$.

\begin{Lemma}[\protect{\cite[Lemma~4]{Gw1}}]\label{L:leading}
Let $h\in\Rn[x,y]$ and $\xi$ be a nonzero vector of $\Rn^2$. 
If $h(x,y)\geq0$ for all $(x,y)\in\Rn^2$, then $h^{\xi}\geq0$ for all $(x,y)\in\Rn^2$. 
\end{Lemma}

For any polynomial $h\in\Rn[x,y]$ we will say that $h$ \textit{changes sign} 
if $h$ attains both positive and negative values. 

\begin{Lemma}\label{L:znak0}
Let $f=x\cdot F(xy)$, $g=y\cdot G(xy)$, where $F(t)$, $G(t)$  are univariate polynomials with real 
coefficients. If the polynomial $F(t)G(t)$ has a nonzero real root, then $\Jac(f,g)$ changes sign. 
\end{Lemma}

\begin{proof} 
It is easy to see that $\Jac(f,g)=H'(xy)$ for $H(t)=tF(t)G(t)$. 
The statement follows from the fact that the derivative of a polynomial 
having two different real roots changes sign. 
\end{proof}

\begin{Lemma}\label{L:znak}
Let $f$, $g\in \Rn[x,y]$ and $\xi=(-1,1)$. 
Assume that $f^{\xi}=y(xy-b)^2$ with $b\neq0$ and $l(\Delta_g, \xi)=-1$. 
Then $\Jac(f,g)$ changes sign. 
\end{Lemma}

\begin{proof} 
The polynomials $f^{\xi}$ and $g^{\xi}$ satisfy the assumptions of Lemma~\ref{L:znak0}.
Hence $\Jac(f^{\xi},g^{\xi})$ changes sign.
Since $\Jac(f^{\xi},g^{\xi})\neq0$, we have $\Jac(f,g)^{\xi}=\Jac(f^{\xi},g^{\xi})$. 
Applying Lemma~\ref{L:leading} to $\Jac(f,g)$ finishes the proof. 
\end{proof}

\begin{Lemma}[\protect{\cite[Corollary~2]{Gw1}}]\label{L:wierzcholki}
Let $f$, $g\in \Rn[x,y]$ and $\xi\in\Rn^2$. 
Assume that $\Delta_f^{\xi}=\{\alpha\}$, $\Delta_g^{\xi}=\{\beta\}$, where $\alpha$ and $\beta$ 
are linearly independent. If $\alpha+\beta$ has an even coordinate then $\Jac(f,g)$ changes sign. 
\end{Lemma}

We say that a nonzero polynomial $f\in\Rn[x,y]$ is \textit{degenerated} on an edge $E$ of 
its Newton polygon if $f|_E$ has a multiple factor which is not divisible by $x$ and $y$. 
Otherwise we say that $f$ is \textit{non-degenerated} on $E$. 
It is easy to check that  $f$ is non-degenerated on every edge $E$ that has no interior lattice points. 

Let $\Delta$ be a Newton polygon. Every edge of a form $\Delta^{\xi}$, where $\xi$ has 
at least one positive coordinate is called an \textit{outer edge} of $\Delta$. 

We say that a nonzero polynomial $f\in\Rn[x,y]$ is \textit{convenient} if some positive powers of
$x$ and $y$ belong to $\supp(f)$.  Observe that this is equivalent to $\Delta_f$ 
intercepting each axis in at least a point away from the origin.

\begin{Lemma}[\protect{\cite[Lemma 3.4]{BO}}]\label{L:inf}
Let $f\in\Rn[x,y]$ be a convenient polynomial without critical points.
If $f$ is non-degenerated on each outer edge of $\Delta_f$ then 
all level sets of $f$ are connected. 
\end{Lemma}

\begin{Lemma}\label{L:edge}
Let $f,g\in\Rn[x,y]$ be such that $\Delta_g\subset\Delta_f$.
Then for all $\mu\in\Rn$ but a finite number we have $\Delta_{f+\mu g}=\Delta_f$.
Moreover if for infinitely many values of $\mu$ the polynomial $f+\mu g$
is degenerated on some edge $E$ of $\Delta_f$, 
then there exists a polynomial $h\in\Rn[x,y]\setminus\Rn$ not divisible by $x$ and by $y$ such 
that $h^2$ is a factor of $f|_E$ and $g|_E$.
\end{Lemma}

\begin{proof}
The first part of the lemma is obvious since  $\supp(f+\mu g)\subset \supp(f)$ and 
a strong inclusion holds only for a finite number of values of $\mu$.

To prove the second part, we may assume, replacing if necessary $g$ by $f+\mu g$ that 
$\Delta_g=\Delta_f$ and that $E=\Delta_f^{\xi}$ where $\xi=(\xi_1,\xi_2)$ has co-prime integer 
coordinates.  
Let $\nu=(\nu_1,\nu_2)$ be a vector with integer coordinates such that  $\xi_1\nu_2-\xi_2\nu_1=1$. Let $\tilde f(u,v)=f|_E(u^{\xi_1}v^{\nu_1},u^{\xi_2}v^{\nu_2})$, 
$\tilde g(u,v)=g|_E(u^{\xi_1}v^{\nu_1},u^{\xi_2}v^{\nu_2})$. 
It is easy to check that 
$\tilde f=u^{l(\Delta_f,\xi)}v^{l(E,\nu)} F(1/v)$, 
$\tilde g=u^{l(\Delta_f,\xi)}v^{l(E,\nu)} G(1/v)$, 
where $F$, $G$ are univariate polynomials with nonzero constant terms. 
The degeneracy of $f+\mu g$ on $E$ implies that the polynomial $F+\mu G$
has a multiple factor. Since this holds for infinitely many values of $\mu$, the polynomials 
$F$ and $G$ have a common factor $H^2$. Then 
the polynomial $h=y^{\xi_1\deg H}H(x^{\xi_2}y^{-\xi_1})$ satisfies the conclusion of the lemma. 
\end{proof}

\begin{Lemma}[\protect{\cite[Corollary 4.2]{Gw2}}] \label{L:hrc}
Let $f,g\in\Rn[x,y]$. Assume that $\Delta_f$ has an outer edge that has endpoints 
$(1,0)$, $(a,b)$ with $a>1$, $b>0$ and has no other lattice points.  
Then $(f,g)$ is not a jacobian pair. 
\end{Lemma}


\section{Proof of Theorem~\ref{Tw:main}}

First we prove Theorem~\ref{Tw:main} under 
the assumption that $\deg g\leq 5$.

\medskip
Suppose that there exists an atypical jacobian pair $(f,g)$
such that $\deg f\leq 5$, $\deg g\leq 5$. 
By~\cite{BO} all jacobian pairs$(f,g)$ with $\deg f\leq 4$ are typical.  
Hence $\deg f=5$, $\deg g=5$. 
Let us denote by  $f^{+}$, $g^{+}$ the leading forms of polynomials $f$, $g$, that is the homogeneous
polynomials of degree 5 which are the symbolic restrictions of $f$, $g$ to the segment $E$ 
with endpoints $(0,5)$, $(5,0)$.  Let $D=\gcd(f^{+}, g^{+})$.   
Then $f_1=f^{+}/D$, $g_1=g^{+}/D$ are co-prime homogeneous polynomials of the same degree. 
By Bertini theorem, for all $\mu\in\Rn$ but a finite number a polynomial $f_1+\mu g_1$ 
is co-prime with $D$ and does not have multiple factors. 

Let $h$ be the product of multiple factors of $D$. It follows from above that replacing 
$(f,g)$ by $(f+\mu g, g)$ with generic $\mu$ we may assume that the polynomials 
$f^{+}/h$ and $h$ are co-prime and $f^{+}/h$ does not have multiple factors. 
Replacing  $(f,g)$ by $(f-f(0,0),g-g(0,0))$ we may assume that 
$f$ and $g$ do not have constant terms. Since $(f,g)$ is a jacobian pair,
$x,y\in \supp(f)\cup\supp(g)$ (otherwise $\Jac(f,g)(0,0)=0$).
Thus replacing once more $(f,g)$ by $(f+\mu_1g, g)$  with suitably chosen constant $\mu_1$,
we may additionally assume that:
\begin{itemize}
\item[(a)] $x,y\in\supp(f)$ and $f$ does not have a constant term,
\item[(b)] $\Delta_g\subset\Delta_f$.

\end{itemize}

Condition~(a) implies, that $f$ is convenient. 
It follows from Corollary~\ref{W:1}, Lemma~\ref{L:inf} 
and Lemma~\ref{L:edge} that for infinitely many $\mu\in\Rn$, including 
$\mu=0$, the polynomial $f+\mu g$ is degenerated on some outer edge 
of $\Delta_f$.

\medskip
The rest of the proof will be case by case analysis with respect to the degree of $h$: 

\medskip
\textbf{(I)} $\deg h=0$. 

Applying a linear change of coordinates we may assume that $x$, $y$ do not divide $f^{+}$. 
Then $\Delta_f$ has only one outer edge with endpoints $(0,5)$, $(5,0)$ and $f$ is non-degenerated 
on this edge. This implies that $(f,g)$ is a typical jacobian pair. We arrived at contradiction. 

\medskip
\textbf{(II)} $\deg h=2$.

Applying a linear change of coordinates we may assume that $h=x^2$
and $y$ does not divide $f^{+}$.
Then the Newton polygon $\Delta_f$ is included in $D_1$
and has an outer edge $E$ (marked in pictures $D_1$ -- $D_5$ in green)
with endpoints $(2,3)$, $(5,0)$. Assumption $\deg h=2$ implies that 
$f$ is non-degenerated on $E$, hence it must be 
degenerated on some other outer edge of its Newton polygon. 

The candidates for $\Delta_f$ are polygons $D_1$ -- $D_5$. 
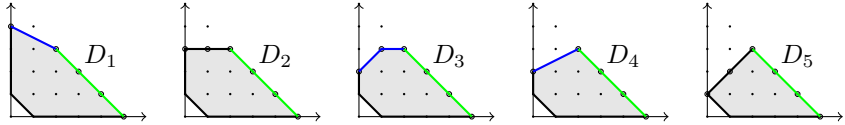
\begin{figure}[h!]
\begin{tikzpicture}[scale = 0.3]
\szkielet
\foreach \position in {(5,0),(4,1),(3,2),(2,3),(0,4)} \draw \position circle (3pt);
    \draw[fill=black, opacity=0.1] (1,0)--(5,0)--(2,3)--(0,4)--(0,1)--cycle;
    \draw[thick] (0,4)--(0,1)--(1,0)--(5,0);
    \draw[thick, color=blue] (2,3)--(0,4);
    \draw[thick, color=green] (5,0)--(2,3);
    \node at (4,2.7) {$D_1$};   
 \end{tikzpicture}
\quad
\begin{tikzpicture}[scale = 0.3]
\szkielet
\foreach \position in {(5,0),(4,1),(3,2),(0,3),(1,3),(2,3)} \draw \position circle (3pt);
    \draw[fill=black, opacity=0.1] (1,0)--(5,0)--(2,3)--(0,3)--(0,1)--cycle;
    \draw[thick] (2,3)--(0,3)--(0,1)--(1,0)--(5,0);
     \draw[thick, color=green] (5,0)--(2,3);
    \node at (4,2.7) {$D_2$};   
 \end{tikzpicture} 
\quad
\begin{tikzpicture}[scale = 0.3]
\szkielet
\foreach \position in {(5,0),(4,1),(0,2),(3,2),(1,3),(2,3)} \draw \position circle (3pt);
    \draw[fill=black, opacity=0.1] (1,0)--(5,0)--(2,3)--(1,3)--(0,2)--(0,1)--cycle;
    \draw[thick] (0,2)--(0,1)--(1,0)--(5,0);
    \draw[thick, color=blue] (2,3)--(1,3)--(0,2);
    \draw[thick, color=green] (5,0)--(2,3);
    \node at (4,2.7) {$D_3$};   
 \end{tikzpicture}
\quad
\begin{tikzpicture}[scale = 0.3]
\szkielet
\foreach \position in {(5,0),(4,1),(0,2),(3,2),(2,3)} \draw \position circle (3pt);
    \draw[fill=black, opacity=0.1] (1,0)--(5,0)--(2,3)--(0,2)--(0,1)--cycle;
    \draw[thick] (0,2)--(0,1)--(1,0)--(5,0);
    \draw[thick, color=blue] (2,3)--(0,2);
    \draw[thick, color=green] (5,0)--(2,3);
    \node at (4,2.7) {$D_4$};   
 \end{tikzpicture}
\quad
\begin{tikzpicture}[scale = 0.3]
\szkielet
\foreach \position in {(5,0),(0,1),(4,1),(1,2),(3,2),(2,3)} \draw \position circle (3pt);
    \draw[fill=black, opacity=0.1] (1,0)--(5,0)--(2,3)--(0,1)--cycle;
    \draw[thick] (2,3)--(0,1)--(1,0)--(5,0);
    \draw[thick, color=green] (5,0)--(2,3);
    \node at (4,2.7) {$D_5$};   
 \end{tikzpicture}
\end{figure}

The possibilities  $D_1$, $D_3$, $D_4$ are excluded since the outer edges different from $E$
do not have interior lattice points which guarantees non-degeneracy. 
That then leave the cases $D_2$ and $D_5$.  If $\Delta_f=D_2$, 
then degeneracy holds on the outer horizontal edge $H$ of $D_2$. 
Therefore $f|_H$ has a form $a(x-b)^2y^3$. 
What is more $(x-b)^2$ divides $g|_H$
Then applying a substitution $x=\bar x +b$ we reduce $\Delta_f$ 
to $D_4$ or $D_5$ and we still can assume, replacing $f$ by $f+\mu g$ if necessary, 
that $\Delta_g\subset\Delta_f$.

Therefore it is enough to consider the case $\Delta_f=D_5$. 
Since there is non-degeneracy on the edge $E$, the polynomial $f$ is
degenerated on the second outer edge $H$ of~$D_5$. 
It follows from Lemma~\ref{L:edge} that $f|_H=ay(xy-b)^2$ 
for some nonzero $a$, $b$ and $g|_H=c f|_H$ for some constant $c$.
Then the symbolic restriction of the polynomial $g-cf$ to $H$ vanishes.  Hence using Lemma~\ref{L:para} and replacing $g$ by $g-cf$
we may assume that $\Delta_g\subset D_6$.  
It is easily seen, that $x\in \supp(g)$, since otherwise 
$(0,0)$ would be a critical point of $g$ 
which is impossible by Corollary~\ref{W:1}. 
If $xy\in\supp(g)$ then $g$ has a critical point $(0,c)$. 
Thus $xy\notin\supp(g)$.
If $x^2y^2\in\supp(g)$, then by Lemma~\ref{L:hrc}
$(f,g)$ is not a jacobian pair.  
If $x^2y^2\notin\supp(g)$, 
then by Lemma~\ref{L:znak} $\Jac(f,g)$ changes sign.  
Summing up $(f,g)$ cannot be a jacobian pair. 

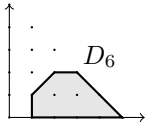
\begin{figure}
\begin{tikzpicture}[scale = 0.3][h!]
\szkielet
    \draw[fill=black, opacity=0.1] (1,0)--(5,0)--(3,2)--(2,2)--(1,1)--cycle;
    \draw[thick] (1,0)--(5,0)--(3,2)--(2,2)--(1,1)--cycle;
    \node at (4,2.7) {$D_6$};   
 \end{tikzpicture}
\end{figure}

\newpage
\textbf{(III)} $\deg h=3$. 

Applying a linear change of coordinates we may assume that $h=x^3$
and $y$ does not divide $f^{+}$.
In the following figure are candidates $D_7$ -- $D_{15}$ for $\Delta_f$. 
The polynomial $f$ is non-degenerated on the edge $E$ marked in green.
The edges marked in blue do not have interior lattice points, 
hence on these edges there is also non-degeneracy. 
By Lemma~\ref{L:inf}, $\Delta_f$ can be only one of polygons 
$D_{12}$ or $D_{13}$ and $f$ is degenerated on the outer horizontal 
edge $H$. By Lemma~\ref{L:edge} the symbolic 
restrictions $f|_H$, $g|_H$ have a factor of a form $(x-a)^2$. 
Then after the substitution $x=\bar x +a$, the Newton polygon of 
a polynomial $f+\mu g$, for generic $\mu$ reduces to $D_{14}$ or $D_{15}$ 
but these possibilities are already excluded. 
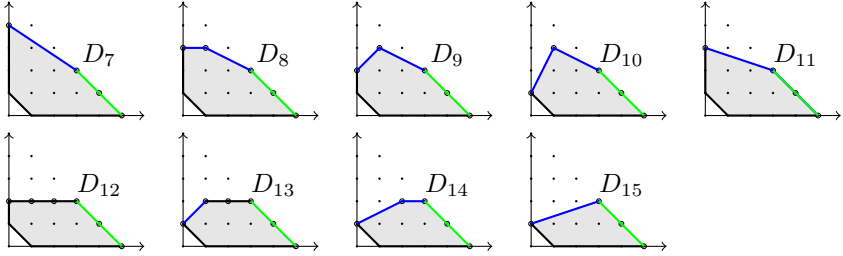
\begin{figure}[h!]
\begin{tikzpicture}[scale = 0.3]
\szkielet
\foreach \position in {(5,0),(4,1),(3,2),(0,4)} \draw \position circle (3pt);
    \draw[fill=black, opacity=0.1] (1,0)--(5,0)--(3,2)--(0,4)--(0,1)--cycle;
    \draw[thick] (0,4)--(0,1)--(1,0)--(5,0);
    \draw[thick, color=blue] (3,2)--(0,4);
    \draw[thick, color=green] (5,0)--(3,2);
    \node at (4,2.7) {$D_7$};   
 \end{tikzpicture}
\quad
\begin{tikzpicture}[scale = 0.3]
\szkielet
\foreach \position in {(5,0),(4,1),(3,2),(0,3),(1,3)} \draw \position circle (3pt);
    \draw[fill=black, opacity=0.1] (1,0)--(5,0)--(3,2)--(1,3)--(0,3)--(0,1)--cycle;
    \draw[thick] (0,3)--(0,1)--(1,0)--(5,0);
    \draw[thick, color=blue] (3,2)--(1,3)--(0,3);
    \draw[thick, color=green] (5,0)--(3,2);
    \node at (4,2.7) {$D_8$};   
 \end{tikzpicture}
\quad
\begin{tikzpicture}[scale = 0.3]
\szkielet
\foreach \position in {(5,0),(4,1),(0,2),(3,2),(1,3)} \draw \position circle (3pt);
    \draw[fill=black, opacity=0.1] (1,0)--(5,0)--(3,2)--(1,3)--(0,2)--(0,1)--cycle;
    \draw[thick] (0,2)--(0,1)--(1,0)--(5,0);
    \draw[thick, color=blue] (3,2)--(1,3)--(0,2);
    \draw[thick, color=green] (5,0)--(3,2);
    \node at (4,2.7) {$D_9$};   
 \end{tikzpicture}
\quad
\begin{tikzpicture}[scale = 0.3]
\szkielet
\foreach \position in {(5,0),(0,1),(4,1),(3,2),(1,3)} \draw \position circle (3pt);
    \draw[fill=black, opacity=0.1] (1,0)--(5,0)--(3,2)--(1,3)--(0,1)--cycle;
    \draw[thick] (0,1)--(1,0)--(5,0);
    \draw[thick, color=blue] (3,2)--(1,3)--(0,1);
    \draw[thick, color=green] (5,0)--(3,2);
    \node at (4,2.7) {$D_{10}$};   
 \end{tikzpicture}
\quad
\begin{tikzpicture}[scale = 0.3]
\szkielet
\foreach \position in {(5,0),(4,1),(3,2),(0,3)} \draw \position circle (3pt);
    \draw[fill=black, opacity=0.1] (1,0)--(5,0)--(3,2)--(0,3)--(0,1)--cycle;
    \draw[thick] (0,3)--(0,1)--(1,0)--(5,0);
    \draw[thick, color=blue] (5,0)--(3,2)--(0,3);
    \draw[thick, color=green] (5,0)--(3,2);
    \node at (4,2.7) {$D_{11}$};   
 \end{tikzpicture}

\vspace{1ex}
\begin{tikzpicture}[scale = 0.3]
\szkielet
\foreach \position in {(5,0),(4,1),(0,2),(1,2),(2,2),(3,2)} \draw \position circle (3pt);
    \draw[fill=black, opacity=0.1] (1,0)--(5,0)--(3,2)--(0,2)--(0,1)--cycle;
    \draw[thick] (3,2)--(0,2)--(0,1)--(1,0)--(5,0);
    \draw[thick, color=green] (5,0)--(3,2);
    \node at (4,2.7) {$D_{12}$};   
 \end{tikzpicture}
\quad
\begin{tikzpicture}[scale = 0.3]
\szkielet
\foreach \position in {(5,0),(0,1),(4,1),(1,2),(2,2),(3,2)} \draw \position circle (3pt);
    \draw[fill=black, opacity=0.1] (1,0)--(5,0)--(3,2)--(1,2)--(0,1)--cycle;
    \draw[thick] (0,1)--(1,0)--(5,0);
    \draw[thick] (3,2)--(1,2);
    \draw[thick, color=green] (5,0)--(3,2);
    \draw[thick, color=blue] (1,2)--(0,1);
    \node at (4,2.7) {$D_{13}$};   
 \end{tikzpicture}
\quad
\begin{tikzpicture}[scale = 0.3]
\szkielet
\foreach \position in {(5,0),(0,1),(4,1),(2,2),(3,2)} \draw \position circle (3pt);
    \draw[fill=black, opacity=0.1] (1,0)--(5,0)--(3,2)--(2,2)--(0,1)--cycle;
    \draw[thick] (0,1)--(1,0)--(5,0);
    \draw[thick, color=blue] (3,2)--(2,2)--(0,1);
    \draw[thick, color=green] (5,0)--(3,2);
    \node at (4,2.7) {$D_{14}$};   
 \end{tikzpicture}
\quad
\begin{tikzpicture}[scale = 0.3]
\szkielet
\foreach \position in {(5,0),(0,1),(4,1),(3,2)} \draw \position circle (3pt);
    \draw[fill=black, opacity=0.1] (1,0)--(5,0)--(3,2)--(0,1)--cycle;
    \draw[thick] (0,1)--(1,0)--(5,0);
    \draw[thick, color=blue] (3,2)--(0,1);
    \draw[thick, color=green] (5,0)--(3,2);
    \node at (4,2.7) {$D_{15}$};   
 \end{tikzpicture}
\end{figure}

\bigskip
\textbf{(IV)} $\deg h=4$. 

Applying a linear change of coordinates we may assume that:
$h=x^4$, $h=x^2y^2$ or $h=h_1^2$,
where $h_1$ is a quadratic form irreducible in $\Rn[x,y]$.

\medskip\noindent
If $h=h_1^2$, then the projective closure of every curve 
$f(x,y)=t$, for $t\in\Rn$ intersects the real part of the line at infinity
at exactly one point with multiplicity 1.
It follows that a curve $f(x,y)=t$ has exactly one real branch at infinity and thus, since all connected components of $f(x,y)=t$ 
are unbounded, we get that $f(x,y)=t$ is connected. 
By Lemma~\ref{L:poziomica} the jacobian pair $(f,g)$ is typical.

\medskip\noindent
Now assume that $h=x^4$. Without loss of generality we may assume that 
$y$ does not divide $f^{+}$. Then $x^5\in\supp(f)$. 
The candidates for $\Delta_f$ are $D_{16}$ -- $D_{24}$. 
It is enough to examine only polygons $D_{20}$, $D_{21}$ and $D_{24}$ 
since in all remaining cases $f$ is non-degenerated on all outer 
edges of its Newton polygon. 

If $\Delta_f=D_{20}$, then $f$ is degenerated on the outer edge 
$H$ with endpoints $(0,3)$, $(4,1)$. 
Hence $f|_H$ has a form $ay(y-bx^2)^2$. 
Then the pair of polynomials 
$(\tilde f, \tilde g)=(f(x,y+b x^2),g(x,y+b x^2))$
is also an atypical jacobian pair. We have either ${\tilde f}^{+}=cx^5$, 
which is the case that we examine in~\textbf{(V)} 
or $\deg \tilde f\leq 4$ which is impossible in of view the results of~\cite{BO}.

If $\Delta_f=D_{21}$, then $f$ is degenerated on the outer horizontal 
edge. Then the suitable substitution of a form $x=\tilde x +a$ 
reduces $\Delta_f$ to $D_{22}$ or $D_{23}$ 
but these possibilities are already excluded. 

If $\Delta_f=D_{24}$, then $f$ is degenerated on the outer horizontal 
edge. Then after some substitution of a form $x=\tilde x+a$ the monomial $y$ no longer belongs to the supports of $f$ and $g$ which implies 
that $\Jac(f,g)(0,0)=0$.

 
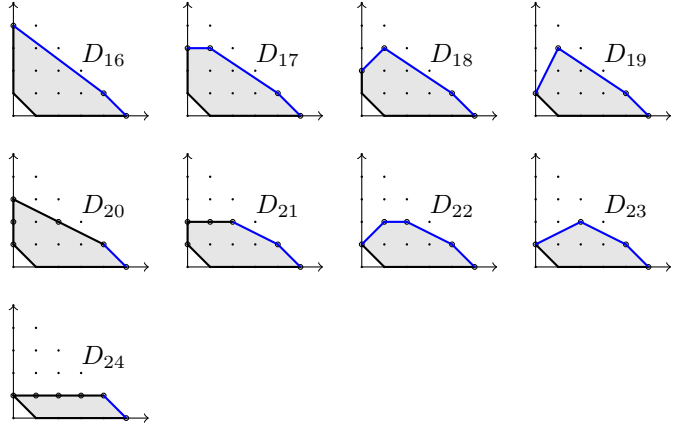
\begin{tikzpicture}[scale = 0.3]
\szkielet
\foreach \position in {(5,0),(4,1),(0,4)} \draw \position circle (3pt);
    \draw[fill=black, opacity=0.1] (1,0)--(5,0)--(4,1)--(0,4)--(0,1)--cycle;
    \draw[thick] (0,4)--(0,1)--(1,0)--(5,0);
    \draw[thick, color=blue] (5,0)--(4,1)--(0,4);
    \node at (4,2.7) {$D_{16}$};   
 \end{tikzpicture}
\quad
\begin{tikzpicture}[scale = 0.3]
\szkielet
\foreach \position in {(5,0),(4,1),(1,3),(0,3)} \draw \position circle (3pt);
    \draw[fill=black, opacity=0.1] (1,0)--(5,0)--(4,1)--(1,3)--(0,3)--(0,1)--cycle;
    \draw[thick] (0,3)--(0,1)--(1,0)--(5,0);
    \draw[thick, color=blue] (5,0)--(4,1)--(1,3)--(0,3);
    \node at (4,2.7) {$D_{17}$};   
 \end{tikzpicture}
\quad
\begin{tikzpicture}[scale = 0.3]
\szkielet
\foreach \position in {(5,0),(4,1),(0,2),(1,3)} \draw \position circle (3pt);
    \draw[fill=black, opacity=0.1] (1,0)--(5,0)--(4,1)--(1,3)--(0,2)--(0,1)--cycle;
    \draw[thick] (0,2)--(0,1)--(1,0)--(5,0);
    \draw[thick, color=blue] (5,0)--(4,1)--(1,3)--(0,2);
    \node at (4,2.7) {$D_{18}$};   
 \end{tikzpicture}
\quad
\begin{tikzpicture}[scale = 0.3]
\szkielet
\foreach \position in {(5,0),(0,1),(4,1),(1,3)} \draw \position circle (3pt);
    \draw[fill=black, opacity=0.1] (1,0)--(5,0)--(4,1)--(1,3)--(0,1)--(0,1)--cycle;
    \draw[thick] (0,1)--(0,1)--(1,0)--(5,0);
    \draw[thick, color=blue] (5,0)--(4,1)--(1,3)--(0,1);
    \node at (4,2.7) {$D_{19}$};   
 \end{tikzpicture}
\quad

\bigskip
\begin{tikzpicture}[scale = 0.3]
\szkielet
\foreach \position in {(5,0),(0,1),(4,1),(0,2),(2,2),(0,3)} \draw \position circle (3pt);
    \draw[fill=black, opacity=0.1] (1,0)--(5,0)--(4,1)--(0,3)--(0,1)--cycle;
    \draw[thick] (4,1)--(0,3)--(0,1)--(1,0)--(5,0);
    \draw[thick, color=blue] (5,0)--(4,1);
    \node at (4,2.7) {$D_{20}$};   
 \end{tikzpicture}
\quad
\begin{tikzpicture}[scale = 0.3]
\szkielet
\foreach \position in {(5,0),(0,1),(4,1),(0,2),(1,2),(2,2)} \draw \position circle (3pt);    
\draw[fill=black, opacity=0.1] (1,0)--(5,0)--(4,1)--(2,2)--(0,2)--(0,1)--cycle;
    \draw[thick] (2,2)--(0,2)--(0,1)--(1,0)--(5,0);
    \draw[thick, color=blue] (5,0)--(4,1)--(2,2);
    \node at (4,2.7) {$D_{21}$};   
 \end{tikzpicture}
\quad
\begin{tikzpicture}[scale = 0.3]
\szkielet
\foreach \position in {(5,0),(0,1),(4,1),(1,2),(2,2)} \draw \position circle (3pt);    
    \draw[fill=black, opacity=0.1] (1,0)--(5,0)--(4,1)--(2,2)--(1,2)--(0,1)--cycle;
    \draw[thick] (0,1)--(1,0)--(5,0);
    \draw[thick, color=blue] (5,0)--(4,1)--(2,2)--(1,2)--(0,1);
    \node at (4,2.7) {$D_{22}$};   
 \end{tikzpicture}
\quad
\begin{tikzpicture}[scale = 0.3]
\szkielet
\foreach \position in {(5,0),(0,1),(4,1),(2,2)} \draw \position circle (3pt);    
    \draw[fill=black, opacity=0.1] (1,0)--(5,0)--(4,1)--(2,2)--(0,1)--cycle;
    \draw[thick] (0,1)--(1,0)--(5,0);
    \draw[thick, color=blue] (5,0)--(4,1)--(2,2)--(0,1);
    \node at (4,2.7) {$D_{23}$};   
 \end{tikzpicture}
\quad

\bigskip
\begin{tikzpicture}[scale = 0.3]
\szkielet
\foreach \position in {(5,0),(0,1),(1,1),(2,1),(3,1),(4,1)} \draw \position circle (3pt);    
    \draw[fill=black, opacity=0.1] (1,0)--(5,0)--(4,1)--(0,1)--cycle;
    \draw[thick] (4,1)--(0,1)--(1,0)--(5,0);
    \draw[thick, color=blue] (5,0)--(4,1);
    \node at (4,2.7) {$D_{24}$};   
\end{tikzpicture}

\medskip\noindent
Finally assume that $h=x^2y^2$. The candidates for $\Delta_f$, 
up to symmetry of the first quadrant, are  $D_{25}$ -- $D_{36}$. 
If the Newton polygon of $f$ has an outer horizontal or vertical edge $E$
and $f$ is degenerated on $E$, then this edge vanishes after 
some substitution $x=\tilde x +a$ if $E$ is horizontal or 
$y=\tilde y +a$ if $E$ is vertical. 
Hence it is enough to consider cases where $\Delta_f$ is one of the 
polygons $D_{33}$ -- $D_{36}$ and $f$ is degenerated on the outer edge 
with endpoint $(0,1)$.  This part of the proof is left to the reader (see the case $\Delta_f=D_5$). 

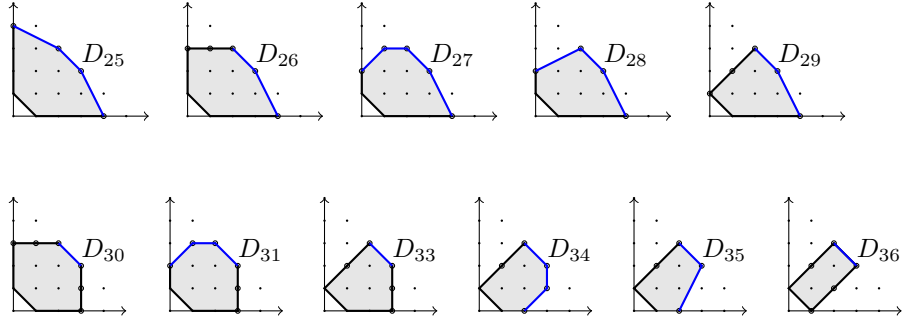
\vspace{1cm}
\begin{tikzpicture}[scale = 0.3]
\szkielet
\foreach \position in {(4,0),(3,2),(2,3),(0,4)} \draw \position circle (3pt);
    \draw[fill=black, opacity=0.1] (0,1)--(1,0)--(4,0)--(3,2)--(2,3)--(0,4)--cycle;
    \draw[thick] (0,4)--(0,1)--(1,0)--(4,0);
    \draw[thick, color=blue] (4,0)--(3,2)--(2,3)--(0,4);
    \node at (4,2.7) {$D_{25}$};   
 \end{tikzpicture}
\quad
\begin{tikzpicture}[scale = 0.3]
\szkielet
\foreach \position in {(4,0),(3,2),(2,3),(0,3),(1,3)} \draw \position circle (3pt);
    \draw[fill=black, opacity=0.1] (0,1)--(1,0)--(4,0)--(3,2)--(2,3)--(0,3)--cycle;
    \draw[thick] (2,3)--(0,3)--(0,1)--(1,0)--(4,0);
    \draw[thick, color=blue] (4,0)--(3,2)--(2,3);
    \node at (4,2.7) {$D_{26}$};   
 \end{tikzpicture}
\quad
\begin{tikzpicture}[scale = 0.3]
\szkielet
\foreach \position in {(4,0),(3,2),(2,3),(1,3),(0,2)} \draw \position circle (3pt);
    \draw[fill=black, opacity=0.1] (0,1)--(1,0)--(4,0)--(3,2)--(2,3)--(1,3)--(0,2)--cycle;
    \draw[thick] (0,2)--(0,1)--(1,0)--(4,0);
    \draw[thick, color=blue] (4,0)--(3,2)--(2,3)--(1,3)--(0,2);
    \node at (4,2.7) {$D_{27}$};   
 \end{tikzpicture}
\quad
\begin{tikzpicture}[scale = 0.3]
\szkielet
\foreach \position in {(4,0),(3,2),(2,3),(0,2)} \draw \position circle (3pt);
    \draw[fill=black, opacity=0.1] (0,1)--(1,0)--(4,0)--(3,2)--(2,3)--(0,2)--cycle;
    \draw[thick] (0,2)--(0,1)--(1,0)--(4,0);
    \draw[thick, color=blue] (4,0)--(3,2)--(2,3)--(0,2);
    \node at (4,2.7) {$D_{28}$};   
 \end{tikzpicture}
\quad
\begin{tikzpicture}[scale = 0.3]
\szkielet
\foreach \position in {(4,0),(3,2),(2,3),(0,1),(1,2)} \draw \position circle (3pt);
    \draw[fill=black, opacity=0.1] (0,1)--(1,0)--(4,0)--(3,2)--(2,3)--cycle;
    \draw[thick] (2,3)--(0,1)--(1,0)--(4,0);
    \draw[thick, color=blue] (4,0)--(3,2)--(2,3);
    \node at (4,2.7) {$D_{29}$};   
 \end{tikzpicture}
\quad

\vspace{1cm}
\begin{tikzpicture}[scale = 0.3]
\szkiel
\foreach \position in {(3,0),(3,2),(2,3),(0,3),(1,3),(3,1)} \draw \position circle (3pt);
    \draw[fill=black, opacity=0.1] (0,1)--(1,0)--(3,0)--(3,2)--(2,3)--(0,3)--cycle;
    \draw[thick] (2,3)--(0,3)--(0,1)--(1,0)--(3,0)--(3,2);
    \draw[thick, color=blue] (3,2)--(2,3);
    \node at (4,2.7) {$D_{30}$};   
 \end{tikzpicture}
\quad
\begin{tikzpicture}[scale = 0.3]
\szkiel
\foreach \position in {(3,0),(3,2),(2,3),(0,2),(1,3),(3,1)} \draw \position circle (3pt);
    \draw[fill=black, opacity=0.1] (0,1)--(1,0)--(3,0)--(3,2)--(2,3)--(1,3)--(0,2)--cycle;
    \draw[thick] (0,2)--(0,1)--(1,0)--(3,0)--(3,2);
    \draw[thick, color=blue] (3,2)--(2,3)--(1,3)--(0,2);
    \node at (4,2.7) {$D_{31}$};   
 \end{tikzpicture}
\quad
\begin{tikzpicture}[scale = 0.3]
\szkiel
\foreach \position in {(3,0),(3,2),(2,3),(1,2),(3,1)} \draw \position circle (3pt);
    \draw[fill=black, opacity=0.1] (0,1)--(1,0)--(3,0)--(3,2)--(2,3)--cycle;
    \draw[thick] (2,3)--(0,1)--(1,0)--(3,0)--(3,2);
    \draw[thick, color=blue] (3,2)--(2,3);
    \node at (4,2.7) {$D_{33}$};   
 \end{tikzpicture}
\quad
\begin{tikzpicture}[scale = 0.3]
\szkiel
\foreach \position in {(2,0),(3,2),(2,3),(1,2),(3,1)} \draw \position circle (3pt);
    \draw[fill=black, opacity=0.1] (0,1)--(1,0)--(2,0)--(3,1)--(3,2)--(2,3)--cycle;
    \draw[thick] (2,3)--(0,1)--(1,0);
    \draw[thick, color=blue] (2,0)--(3,1)--(3,2)--(2,3);
    \node at (4,2.7) {$D_{34}$};   
 \end{tikzpicture}
\quad
\begin{tikzpicture}[scale = 0.3]
\szkiel
\foreach \position in {(2,0),(3,2),(2,3),(1,2)} \draw \position circle (3pt);
    \draw[fill=black, opacity=0.1] (0,1)--(1,0)--(2,0)--(3,2)--(2,3)--cycle;
    \draw[thick] (2,3)--(0,1)--(1,0);
    \draw[thick, color=blue] (2,0)--(3,2)--(2,3);
    \node at (4,2.7) {$D_{35}$};   
 \end{tikzpicture}
\quad
\begin{tikzpicture}[scale = 0.3]
\szkiel
\foreach \position in {(1,0),(2,1),(3,2),(2,3),(1,2)} \draw \position circle (3pt);
    \draw[fill=black, opacity=0.1] (0,1)--(1,0)--(3,2)--(2,3)--cycle;
    \draw[thick] (2,3)--(0,1)--(1,0)--(3,2)--(2,3);
    \draw[thick, color=blue] (3,2)--(2,3);
    \node at (4,2.7) {$D_{36}$};   
 \end{tikzpicture}
\quad

\bigskip
\textbf{(V)} $\deg h=5$. 
In this case we have $g=cf^{+}+\mbox{terms of degree~$\leq 4$}$. 
Then $(g-cf,f)$ is an atypical jacobian pair with first component of degree at most 4 
which is impossible in view of the result of Braun and Or\'efice-Okamoto~\cite{BO}.


\medskip
In order to complete the proof it is enough to show the following results.

\begin{Theorem}\label{Tw:deg}
Every jacobian pair $(f,g)$ such that 
$f=x^5+\mbox{terms of degree $\leq 4$}$
is typical. 
\end{Theorem}

\begin{proof}
Assume that there exists an atypical jacobian pair $(f,g)$ 
satisfying hypothesis of the theorem. Then the pair of polynomials 
$(\tilde f,\tilde g)=(f(x-a,y)+b,g(x-a,y))$ for any $a,b\in\Rn$ 
is also an atypical jacobian pair. Thus, without loss of generality
we may assume, 
replacing $(f,g)$ by $(\tilde f,\tilde g)$ if necessary, 
that $f$ has a nonzero constant term
and $\Delta_f$ does not have an outer edge of positive slope. 

\medskip
The candidates for $\Delta_f$ are polygons $D_{37}$ -- $D_{46}$. 

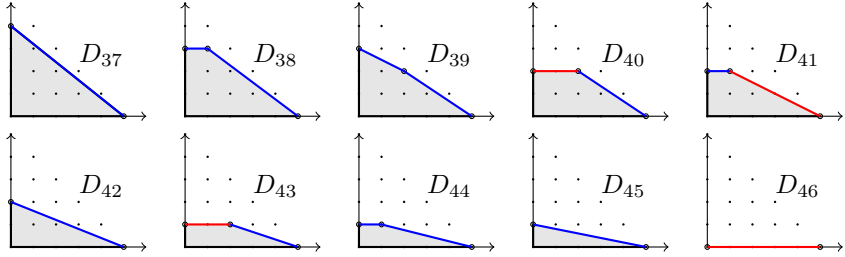
\begin{figure}[h!]
\begin{tikzpicture}[scale = 0.3]
\szkielet
\foreach \position in {(5,0),(0,4)} \draw \position circle (3pt);
    \draw[fill=black, opacity=0.1] (5,0)--(0,4)--(0,0)--cycle;
    \draw[thick] (5,0)--(0,4)--(0,0)--(5,0);
    \draw[thick, color=blue] (5,0)--(0,4);
    \node at (4,2.7) {$D_{37}$};   
 \end{tikzpicture}
\quad
\begin{tikzpicture}[scale = 0.3]
\szkielet
\foreach \position in {(5,0),(1,3),(0,3)} \draw \position circle (3pt);
    \draw[fill=black, opacity=0.1] (5,0)--(1,3)--(0,3)--(0,0)--cycle;
    \draw[thick] (0,3)--(0,0)--(5,0);
    \draw[thick, color=blue] (5,0)--(1,3)--(0,3);
    \node at (4,2.7) {$D_{38}$};   
 \end{tikzpicture} 
\quad
\begin{tikzpicture}[scale = 0.3]
\szkielet
\foreach \position in {(5,0),(2,2),(0,3)} \draw \position circle (3pt);
    \draw[fill=black, opacity=0.1] (5,0)--(2,2)--(0,3)--(0,0)--cycle;
    \draw[thick] (0,3)--(0,0)--(5,0);
    \draw[thick, color=blue] (5,0)--(2,2)--(0,3);
    \node at (4,2.7) {$D_{39}$};   
 \end{tikzpicture} 
\quad
\begin{tikzpicture}[scale = 0.3]
\szkielet
\foreach \position in {(5,0),(2,2),(0,2)} \draw \position circle (3pt);
    \draw[fill=black, opacity=0.1] (5,0)--(2,2)--(0,2)--(0,0)--cycle;
    \draw[thick] (0,2)--(0,0)--(5,0);
    \draw[thick, color=blue] (5,0)--(2,2);
        \draw[thick, color=red] (2,2)--(0,2);
    \node at (4,2.7) {$D_{40}$};   
 \end{tikzpicture} 
\quad
\begin{tikzpicture}[scale = 0.3]
\szkielet
\foreach \position in {(5,0),(1,2),(0,2)} \draw \position circle (3pt);
    \draw[fill=black, opacity=0.1] (5,0)--(1,2)--(0,2)--(0,0)--cycle;
    \draw[thick] (0,2)--(0,0)--(5,0);
    \draw[thick, color=blue] (1,2)--(0,2);
    \draw[thick, color=red] (5,0)--(1,2);
    \node at (4,2.7) {$D_{41}$};   
 \end{tikzpicture} 

\vspace{1ex}
\begin{tikzpicture}[scale = 0.3]
\szkielet
\foreach \position in {(5,0),(0,2)} \draw \position circle (3pt);
    \draw[fill=black, opacity=0.1] (5,0)--(0,2)--(0,0)--cycle;
    \draw[thick] (0,2)--(0,0)--(5,0);
    \draw[thick, color=blue] (5,0)--(0,2);
    \node at (4,2.7) {$D_{42}$};   
\end{tikzpicture}     
\quad
\begin{tikzpicture}[scale = 0.3]
\szkielet
\foreach \position in {(5,0),(2,1),(0,1)} \draw \position circle (3pt);
    \draw[fill=black, opacity=0.1] (5,0)--(2,1)--(0,1)--(0,0)--cycle;
    \draw[thick] (0,1)--(0,0)--(5,0);
    \draw[thick, color=blue] (5,0)--(2,1);
     \draw[thick, color=red] (2,1)--(0,1);
    \node at (4,2.7) {$D_{43}$};   
\end{tikzpicture} 
\quad
\begin{tikzpicture}[scale = 0.3]
\szkielet
\foreach \position in {(5,0),(1,1),(0,1)} \draw \position circle (3pt);
    \draw[fill=black, opacity=0.1] (5,0)--(1,1)--(0,1)--(0,0)--cycle;
    \draw[thick] (0,1)--(0,0)--(5,0);
    \draw[thick, color=blue] (5,0)--(1,1)--(0,1);
    \node at (4,2.7) {$D_{44}$};   
\end{tikzpicture} 
\quad
\begin{tikzpicture}[scale = 0.3]
\szkielet
\foreach \position in {(5,0),(0,1)} \draw \position circle (3pt);
    \draw[fill=black, opacity=0.1] (5,0)--(0,1)--(0,0)--cycle;
    \draw[thick] (0,1)--(0,0)--(5,0);
    \draw[thick, color=blue] (5,0)--(0,1);
    \node at (4,2.7) {$D_{45}$};   
\end{tikzpicture} 
\quad
\begin{tikzpicture}[scale = 0.3]
\szkielet
\foreach \position in {(5,0),(0,0)} \draw \position circle (3pt);
    \draw[thick, color=red] (5,0)--(0,0);
    \node at (4,2.7) {$D_{46}$};   
\end{tikzpicture} 
\end{figure}

In the above figures the outer edges without  interior lattice points are marked in blue.

If $\Delta_f$ is one of polygons $D_i$ for $i\in\{37,38,39,42,44,46\}$, 
then $f$ is non-degenerated on each outer edge of $\Delta_f$. Then by 
Lemma~\ref{L:inf} and Lemma~\ref{L:poziomica} $(f,g)$ is a typical jacobian pair. 

If $\Delta_f=D_{40}$, then $f$ is degenerated on the horizontal outer edge $E$ marked in red. 
The symbolic restriction $f|_E$ has a factor of a form $(x-a)^2$.  Then after the substitution 
$x=\tilde x +a$ and after subtracting the constant term the Newton polygon 
$\Delta_f$ reduces to $D_{48}$ or $D_{49}$.  If $\Delta_f=D_{48}$ then $f$ has a critical 
point of a form $(0,c)$. If If $\Delta_f=D_{49}$ then $(f,g)$ satisfies the 
assumptions of Lemma~\ref{L:hrc}. Hence in both cases $(f,g)$ is not a jacobian pair.

If $\Delta_f=D_{41}$, then $f$ is degenerated on the outer edge $E$ marked in red. 
We have $f|_E=ax(y-bx^2)^2$ for some nonzero constants $a$, $b$. 
Then the pair of polynomials $(\tilde f, \tilde g)=(f(x,y+b x^2),g(x,y+b x^2))$
is an atypical jacobian pair such that  $\deg \tilde f= 4$ which is impossible 
in of view the results of~\cite{BO}.

If $\Delta_f=D_{43}$, then $f$ is degenerated on the horizontal outer edge $E$ marked in red. 
Then after some substitution of the form $x=\tilde x +a$ and after subtracting the constant term the Newton polygon $\Delta_f$ reduces to $D_{50}$. By Lemma~\ref{L:hrc} $(f,g)$ 
is not a jacobian pair.

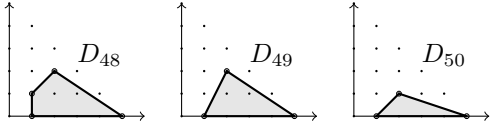
\begin{figure}[h!]
\begin{tikzpicture}[scale = 0.3]
\szkielet
\foreach \position in {(5,0),(2,2),(1,1),(1,0)} \draw \position circle (3pt);
    \draw[fill=black, opacity=0.1] (5,0)--(2,2)--(1,1)--(1,0)--cycle;
    \draw[thick] (5,0)--(2,2)--(1,1)--(1,0)--(5,0);
    \node at (4,2.7) {$D_{48}$};   
 \end{tikzpicture} 
\quad
\begin{tikzpicture}[scale = 0.3]
\szkielet
\foreach \position in {(5,0),(2,2),(1,0)} \draw \position circle (3pt);
    \draw[fill=black, opacity=0.1] (5,0)--(2,2)--(1,0)--cycle;
    \draw[thick] (5,0)--(2,2)--(1,0)--(5,0);
    \node at (4,2.7) {$D_{49}$};   
 \end{tikzpicture} 
\quad
\begin{tikzpicture}[scale = 0.3]
\szkielet
\foreach \position in {(5,0),(2,1),(1,0)} \draw \position circle (3pt);
    \draw[fill=black, opacity=0.1] (5,0)--(2,1)--(1,0)--cycle;
    \draw[thick] (5,0)--(2,1)--(1,0)--(5,0);
    \node at (4,2.7) {$D_{50}$};   
\end{tikzpicture} 
\end{figure}

If $\Delta_f=D_{46}$, then $f$ depends only on variable $x$. Since $f$ does not have critical points,
the level sets $f^{-1}(t)$ are vertical lines,  in particular they are connected. 
It follows from Lemma~\ref{L:poziomica} that $(f,g)$ is a typical jacobian par. 

We checked all possible cases. Hence $(f,g)$ cannot be an atypical jacobian pair. 
\end{proof}

\begin{Theorem}\label{Tw:st6}
Every jacobian pair $(f,g)$ such that $\deg f=5$, $\deg g=6$
is typical. 
\end{Theorem}

\begin{proof} 
Let $(f,g)$ be a jacobian pair satisfying the assumptions of the theorem. 
Since $f^{+}$ has an odd degree, it has a linear factor. Hence applying 
a linear change of coordinates we may assume that $x$ is a factor $f^{+}$. 
Let $k$ be the biggest integer such that $x^k$ divides $f^{+}$.

If $k=5$ then by Theorem~\ref{Tw:deg} $(f,g)$ is a typical jacabian pair. 

If $k\in\{1,2,3,4\}$, then $\alpha=(k,5-k)$, 
is a vertex of $\Delta_f$ and 
for every vector $\xi$ of the form $(n,n+1)$ with $n\geq 5$
we have $\Delta_f^{\xi}=\{\alpha\}$. 
Let $l=\min\{i: x^iy^{6-i}\in\supp(g)\}$. The point $\beta=(l,6-l)$ is 
a vertex of $\Delta_g$. 
Moreover for some vector $\xi=(n,n+1)$ with $n\geq 5$ we have   
$\Delta_g^{\xi}=\{\beta\}$. Since $\alpha+\beta=(k+l,11-k-l)$, we get 
that one coordinate of $\alpha+\beta$ is even. 
Then by Lemma~\ref{L:wierzcholki} the Jacobi determinant $\Jac(f,g)$
changes sign. Hence this case is impossible. 
\end{proof}

\end{document}